\documentclass[reqno]{amsart}

\usepackage{amsmath}
\usepackage{amsfonts}
\usepackage{amsthm}
\usepackage{amssymb}
\usepackage{graphicx}
\usepackage{graphics}
\usepackage{bm}
\usepackage{dsfont}
\usepackage{color}
\usepackage{enumerate}
\usepackage[font=footnotesize]{caption}
\usepackage[all]{xy}
\numberwithin{equation}{section}

\newtheoremstyle{personal}%
{12pt}
{12pt}
{\slshape}
{}
{\bfseries}
{.}
{.5em}
{}
\theoremstyle{personal}%
\newtheorem{thm}{Theorem}[section]
\newtheorem{cor}[thm]{Corollary}
\newtheorem{lem}[thm]{Lemma}
\newtheorem{prop}[thm]{Proposition}
\theoremstyle{definition}

\newtheorem{dfn}[thm]{Definition}
\theoremstyle{remark}
\newtheorem{rem}[thm]{Remark}
\newcommand{\N}{\mathds{N}}
\newcommand{\Z}{\mathds{Z}}
\newcommand{\R}{\mathds{R}}
\newcommand{\C}{\mathds{C}}
\newcommand{\T}{\mathds{T}}

\newcommand{\V}{\mathcal{V}}

\newcommand{\PP}{\mathbb{P}}
\newcommand{\ps}{\partial_s}
\newcommand{\<}{\langle}
\renewcommand{\>}{\rangle}
\newcommand{\na}{\nabla}
\renewcommand{\S}{\mathbb{S}}
\newcommand{\A}{\mathbb{A}}
\newcommand{\pt}{\partial_t}

\newcommand{\vp}{\varphi}

\newcommand{\red}{/\!\!/}

\DeclareMathOperator{\ad}{ad}
\DeclareMathOperator{\Ad}{Ad}
\DeclareMathOperator{\rk}{rk}
\newcommand{\h}{\mathfrak{h}}
\newcommand{\g}{\mathfrak{g}}

\newcommand{\hor}{\mathrm{hor}}

\begin{document}

\title[Symplectic reduction and magnetic flows]{On the existence of closed magnetic geodesics via symplectic reduction}

\author[L. Asselle]{Luca Asselle}
\address{Ruhr Universit\"at Bochum, Fakult\"at f\"ur Mathematik, Universit\"atsstra\ss e 150 \newline\indent Geb\"aude NA 4/35, D-44801 Bochum, Germany}
\email{luca.asselle@ruhr-uni-bochum.de}

\author[F. Schm\"aschke]{Felix Schm\"aschke}
\address{Humboldt-Universit\"at zu Berlin, Mathematisch-Naturwissenschaftliche Fakult\"at, \newline\indent Unter den Linden 6, D-10099 Berlin, Germany}
\email{felix.schmaeschke@math.hu-berlin.de}

\date{July 20, 2016}

\begin{abstract}
Let $(M,g)$ be a closed Riemannian manifold and $\sigma$ be a closed 2-form on $M$ representing an integer cohomology class.
In this paper, using symplectic reduction, we show how the problem of existence of closed magnetic geodesics for the magnetic flow of the pair $(g,\sigma)$  
can be interpreted as a critical point problem for a Rabinowitz-type action functional defined on the cotangent bundle $T^*E$ of a suitable 
$S^1$-bundle $E$ over $M$ or, equivalently, as a critical point problem for a Lagrangian-type action functional defined on the free loopspace of $E$. 
We then study the relation between the stability property of energy hypersurfaces in $(T^*M,dp\wedge dq+\pi^*\sigma)$ and of the corresponding codimension 2 coisotropic submanifolds
in $(T^*E,dp\wedge dq)$ arising via symplectic reduction. Finally, we reprove the main result of \cite{Asselle:2014hc} in this setting. 
\tableofcontents
\end{abstract}

\maketitle
\vspace{-15mm}
\section{Introduction}
\label{s:introduction}

Let $(M,g)$ be a closed Riemannian manifold and let
$\sigma$ be a closed 2-form on $M$. Up to passing to the orientable double cover of $M$ 
we can suppose without loss of generality that $M$ is orientable. Consider the kinetic Hamiltonian
$$ \bar H: T^*M\rightarrow \R\, ,\qquad \bar H(\bar q,\bar p) = \frac 12 |\bar p|_{\bar q}^2\, ,$$
where as usual $|\cdot|$ denotes the (dual) norm on $T^*M$ induced by
the metric $g$. Consider also the twisted symplectic form
$\bar \omega_\sigma=\bar \omega + \bar \pi^*\sigma$, where $\bar \omega=d\bar p\wedge d\bar q$ is
the canonical symplectic form on $T^*M$ and $\bar \pi:T^*M \to M$ the
canonical projection. The pair $(\bar H,\bar \omega_\sigma)$ defines a vector
field $X_{\bar H}^\sigma$ on $T^*M$ by
$$\bar \omega_\sigma \big (  X_{\bar H}^\sigma,\cdot ) = -d\bar H,$$
called the Hamiltonian vector field of $\bar H$ with respect to
$\bar \omega_\sigma$. Its flow $ \Phi^\sigma_{\bar H}:T^*M\rightarrow T^*M$ is
the \textit{magnetic flow of the pair $(g,\sigma)$}. The reason
of this terminology is that it models the motion of a charged
particle in $M$ under the effect of a magnetic field
represented by $\sigma$. In fact, $x:I \to T^*M$ is a flow
line of $X^\sigma_{\bar H}$ if and only if the curve $\mu = \pi \circ x$ satisfies
the second order ordinary differential equation
\[
\nabla_t \dot \mu = Y_\mu(\dot \mu)\,,
\]
where $\nabla_t$ denotes the covariant derivative associated to $g$
and $Y:TM \to TM$ is the linear bundle map (known as \textit{Lorentz force}) given by
\[
g_q(u,Y_{\bar q}(v)) = \sigma_{\bar q}(u,v), \qquad \forall\ u,v \in T_{\bar q} M,\ \forall \bar q \in M.
\]

Periodic orbits of such a flow are usually called \textit{closed magnetic geodesics}. The magnetic flow preserves $\bar H$, since it is the Hamiltonian of the system; therefore it makes sense
to look at periodic orbits on a given level set. In this paper we will
be interested in the following problem: given $\bar k> 0$, does there
exist a period $T>0$ and a curve $x:\R \rightarrow T^*M$ which satisfies
the following conditions?
\begin{equation}
  \left \{\begin{array}{l}
      \dot x (t) = X_{\bar H}^\sigma (x(t))\,;\\ 
      x(T)=x(0)\,;\\
      \bar H(x)=\bar k\, .
\end{array}\right.
\label{eq:I}
\end{equation}

A particular case of magnetic flow is given by the choice $\sigma=0$, in which case we retrieve the \textit{geodesic flow} of $(M,g)$. 
The problem of the existence of closed geodesics has received in the last century the attention of many
outstanding mathematicians as Birkhoff, Lyusternik, Gromoll and Meyer, just to mention few of them. 
The existence of periodic orbits for magnetic flows represents a
natural generalization of the closed geodesic problem. However, unlike
the geodesic case, the dynamics in the magnetic setting turns out to depend essentially on the
kinetic energy of the particle. This is one of the reason why
existence results for closed geodesics cannot be straightforward
generalized to the magnetic setting. In fact, Hedlund
\cite{Hedlund:1932am} provided an example of a ``critical'' energy
level without closed magnetic geodesics on any surface with genus at
least two. On the other hand, almost every energy level contains at
least one closed magnetic geodesic (c.f. \cite{Asselle:2014hc} and
references therein). 

In the literature various approaches and techniques, coming for instance from the classical calculus of variations \cite{Abbondandolo:2013is,Abbondandolo:2014rb,Abbondandolo:2015lt,Asselle:2015ij,Contreras:2006yo,Merry:2010}, 
symplectic geometry \cite{Cieliebak:2010zt,Frauenfelder:2007le,Ginzburg:1987lq,Ginzburg:1994,Ginzburg:2009,Paternain:2009,Schlenk:2006vk,Usher:2009}, symplectic homology \cite{Benedetti:2014} and contact homology \cite{Ginzburg:2014ws}, are used to tackle the problem of existence of closed magnetic geodesics. 
See also \cite{Asselle:2016qv,Taimanov:1991el,Taimanov:1992fs,Taimanov:1992sm} for existence results based on a minimization procedure in case the configuration
space is two-dimensional.
In particular, for magnetic flows defined by an exact 2-form $\sigma=d\theta$ the existence of closed magnetic geodesics can be shown
by using a variational characterization of periodic orbits as critical points of the free-period Lagrangian action functional (see e.g. \cite{Abbondandolo:2013is,Contreras:2006yo}). 
If one tries to generalize this approach dropping the exactness assumption, then one has to overcome the difficulty given by the fact that the action functional is not well-defined but rather ``multi-valued".
Nevertheless, following ideas contained in \cite{Novikov:1982,Novikov:1984hc,Taimanov:1983uo}, progresses in this direction have been recently made in \cite{Asselle:2014hc,Asselle:2015sp} by studying the existence of zeros of the 
action 1-form.

In this paper we use another approach to study the existence of
solutions to \eqref{eq:I} based on the following remark: the
twisted cotangent bundle arises naturally via symplectic reduction
(c.f. \cite[Ex. 5.2]{ReymanSemenov} or \cite[Section 6.6]{OrtegaRatiu}). 
If $\sigma$ represents an integer cohomology class, then this allows to interpret the magnetic flow as a geodesic flow on the cotangent bundle of a suitable
$S^1$-bundle $E$ over $M$, at the cost of introducing a symmetry
group. In particular, closed magnetic geodesics with energy $\bar k$
turn out to correspond to the critical points of a Rabinowitz-type action
functional 
$$\A_k:C^\infty(S^1,T^*E)\times (0,+\infty)\times \R\to \R$$
or equivalently, using the Legendre transform, to the critical points of a Lagrangian-type action functional
$$\S_k:H^1(S^1,E)\times (0,+\infty)\times \R\to \R.$$

Here $k=\bar k + \frac 12$ and $H^1(S^1,E)$ denotes the Hilbert manifold of absolutely continuous loops in $E$ with square-integrable derivative. 
Notice that the correspondence between closed magnetic geodesics and critical points of $\A_k$ would allow to use 
a version of Rabinowitz-Floer homology for contact type (or, at least, stable) coisotropic submanifolds - as developed by Kang in \cite{Kang2013} - to infer  existence on a given energy level.
To this purpose, it is important to study the stability property of such coisotropic submanifolds, also in relation with the stability property of the corresponding hypersurfaces in $T^*M$. This will be carried over in Section 
\ref{stabilityandcontactpropertyofcoisotropicsubmanifolds}, where we also provide some concrete examples. 
In the last part of the paper, building on the latter correspondence, we reprove the main theorem of \cite{Asselle:2014hc} in the setting of magnetic flows given by closed 2-forms representing an integer cohomology
class. 

\begin{thm} 
Let $(M,g)$ be a closed non-aspherical Riemannian manifold, i.e. 
$\pi_\ell(M)\neq 0$ for some $\ell\geq2$, and $\sigma$ be a closed 2-form on $M$ representing an integer cohomology class. Then for almost every $\bar k>0$ there 
exists a contractible closed magnetic geodesic with energy $\bar k$.
\label{thm:main}
\end{thm}

We end this introduction by giving a summary of the contents of this paper:
In Section \ref{s:symplecticreduction} we recall how the magnetic flow can be seen as a projected geodesic flow and 
introduce the functional $\A_k$.
 In Section \ref{stabilityandcontactpropertyofcoisotropicsubmanifolds} we  discuss the relation 
between stability and contact property of energy hypersurfaces and of the corresponding coistropic submanifolds arising via symplectic reduction.
In Section \ref{thelagrangianactionfunctional} we introduce the functional $\S_k$ and study its properties.
 In Section \ref{proofoftheorem} we prove Theorem \ref{thm:main}.


\section{Symplectic reduction}
\label{s:symplecticreduction}

\subsection{The magnetic flow as a projected geodesic flow.}Let
$(M,g)$ be a closed orientable Riemannian manifold and let $\sigma$ be
a closed 2-form on $M$. We call the pair $(T^*M,\bar \omega_\sigma:=d\bar p\wedge d\bar q +
\bar \pi^*\sigma)$ the \emph{twisted cotangent bundle}. It has been known
for a long time that twisted cotangent bundles arise via symplectic
reduction (c.f.\ for example \cite[Ex. 5.2]{ReymanSemenov}). Here we
quickly recall this construction.

Throughout this paper we assume that the deRahm cohomology
class represented by $\sigma$ is integral, i.e.\ $[\sigma]\in
H^2(M;\Z)$. Let $S^1=\{ e^{i t} \in \C \mid t \in \R\}$ be the Lie group of
complex numbers of norm one. If $\sigma$ represents an integral
cohomology class, then there is a principal $S^1$-bundle $\tau:E\rightarrow
M$ with Euler class $e(E)=[\sigma]\in H^2(M;\Z)$. 

Recall that the Euler class is constructed as follows: choose a \emph{connection
  $1$-form} $\theta\in \Omega^1(E)$, which is an $S^1$-invariant
$1$-form satisfying $\theta(Z)=1$, where $Z$ denotes the fundamental
vector field of the $S^1$-action
$$Z_q=\frac{d}{dt}\, e^{it}q\Big |_{t=0} \in T_q E, \qquad \forall q \in E\,.$$ 
The form $\theta$ induces a splitting of the tangent bundle 
\begin{equation}
  \label{eq:split}
TE  = \ker \theta \oplus \R\! \cdot\!Z,  
\end{equation}
(vectors in $\ker \theta$ are called \emph{horizontal}), and uniquely defines a \emph{curvature form} $\tilde \sigma \in \Omega^2(M)$ by
\[
\tilde \sigma_{\bar q}(u,v)= (d\theta)_q(u^\hor,v^\hor)\,,
\]
where $u,v \in T_{\bar q} M$, $\bar q \in M$, $q \in \tau^{-1}(\bar q)$ and
$u^\hor, v^\hor \in T_q E$ are horizontal vectors that
project to $u,v$ via $d_q \tau$ respectively (called \emph{horizontal lift}). Obviously $d
\tilde \sigma=0$. The \textit{Euler class} is defined as the cohomology class
represented by $\tilde \sigma$. To see that $[\tilde \sigma]$ does not depend on the
choice of $\theta$, one shows that any another connection form
$\theta'$ must satisfy $\theta' = \theta+ \tau^*\beta$ for some
$\beta \in \Omega^1(M)$. The curvature of $\theta'$ is therefore
$\tilde \sigma + d \beta$ and hence defines the same cohomology class.
Notice that this also shows that the map $\theta\mapsto \tilde \sigma$ from the space of connection 1-forms to the space of closed forms on $M$ representing the cohomology class $e(E)$ is surjective.
In particular, for a given closed 2-form $\sigma$ on $M$ representing an integer cohomology class we can always find a connection 1-form $\theta$ such that $d\theta=\tau^*\sigma$. 

By push-forward  the $S^1$-action on $E$ lifts canonically to an $S^1$-action on $T^*E$
$$T^*E \rightarrow T^*E,\quad (q,p) \mapsto \left (e^{it}q, p\cdot (d_qe^{it})^{-1}\right).$$
 It is a classical fact (see for instance \cite{Abraham:1978mf}) that this action on $T^*E$ is the Hamiltonian flow with
respect to the standard symplectic structure of the Hamiltonian
$$A:T^*E\longrightarrow \R\, ,\qquad (q,p)\longmapsto \<p,Z_q\>\,. $$
Since the action is free, for every $c\in \R$ the symplectic quotient
is well-defined
\[
T^*E\red_c\, S^1 := A^{-1}(c)/S^1\,.
\]
This quotient manifold is naturally endowed with a symplectic form
$\bar \omega_c$, which is defined as the unique form such that $\mathrm{pr}^*
\bar \omega_c = \imath^*\omega$, where $\imath:A^{-1}(c)\hookrightarrow
T^*E$, $\mathrm{pr}:A^{-1}(c)\rightarrow T^*E\red_c\,S^1$ and $\omega$ denote
respectively the natural inclusion, the projection map and the
standard symplectic form on $T^* E$. Fix a connection form $\theta$ and define a map $\Pi_c:A^{-1}(c) \to T^*M$ implicitly via 
\begin{equation}\label{eq:Pic}
\<\Pi_c(q,p), d_q \tau\, v\>= \<p,v\> - c\, \theta(v),\qquad \forall\ v \in T_q E\,.
\end{equation}
Note that $\Pi_c$ is well-defined because the kernel of $d_q \tau$ is spanned precisely by the fundamental vector field, on which the right-hand side vanishes. Moreover it is not hard to see that $\Pi_c$ is a bundle map with fibres consisting of $S^1$-orbits for the lifted $S^1$-action. We conclude that $\Pi_c$ induces a diffeomorphism $T^*E\red_c S^1 \cong T^*M$.

\begin{prop}\label{prp:red}
For all $c \in \R$ the map $\Pi_c$ induces a symplectomorphism 
\[
(T^*E\red_c S^1,\bar \omega_c) \cong (T^*M, \bar \omega + c \bar \pi^*\sigma)\,.
\]
\end{prop}
\begin{proof}
We need to show that $\Pi_c^*(\bar \omega + c\bar \pi^*\sigma) = i^* \omega$. Since we have $\bar \pi\circ \Pi_c=\tau \circ \pi$ we conclude that 
\[
\Pi_c^*\bar \pi^* \sigma = \pi^* \tau^* \sigma = \pi^* d \theta=d\pi^*\theta\,.
\]
Hence it suffices to see that $\Pi^*_c\bar \lambda + c\pi^* \theta = i^*\lambda$, where $\bar \lambda$, $\lambda$ are the Liouville forms in $T^*M$ and $T^*E$ respectively. For any $v \in T_{(q,p)}A^{-1}(c)$ we denote $(\bar q,\bar p) = \Pi_c(q,p)$ and compute
\[
(\Pi^*_c \bar \lambda)_{(q,p)} (v) = \< \bar p, d\bar \pi d \Pi_\theta v\> = \<\bar p, d\tau d\pi v\>\,,\]
and using the definition~\eqref{eq:Pic} we continue the computation
\[
(\Pi^*_c \bar \lambda)_{(q,p)} (v)
 = 
\<p,d\pi v\>-c\theta(d\pi v) = \lambda_{q,p}(v) - c(\pi^*\theta)_{q,p}(v)\,.
\]
This shows the claim. 
\end{proof}

Fix a connection form $\theta$ for $\sigma$ and lift the metric on $M$ to a metric
on $E$ via $g^\theta:=\tau^*g + \theta\otimes \theta$. In other words,
consider the unique metric on $E$ such that:
\begin{itemize}
\item $d_q \tau:\ker \theta_q \to T_{\tau q}M$ is an isometry for all $q\in E$.
\item $g^\theta(X,X) = 1$,
\item the splitting~\eqref{eq:split} is orthogonal.
\end{itemize}

By abuse of notation we denote the (dual) norm on $T^*E$ induced by
$g^\theta$ again with $|\,\cdot\,|$ and the kinetic Hamiltonian again with
$$H:T^*E\longrightarrow \R\, ,\qquad H(q,p)=\frac 12 |p|_q^2 .$$

Since by construction the metric $g^\theta$ is $S^1$-invariant the
Hamiltonian flow of $H$ commutes with the Hamiltonian flow of $A$. In particular the flow of $H$ preserves the levels of $A$ and via Proposition~\ref{prp:red} projects to a Hamiltonian flow on $(T^*M,\bar \omega_\sigma)$. We show now that this reduced flow is precisely the magnetic flow. 
\begin{lem} \label{lem:sympreduction} We have $H=\bar H \circ \Pi_1 + \frac 12$ and  $d\Pi_1 X_H = X^\sigma_{\bar H}$. In particular, a curve $\bar x:\R\rightarrow
  T^*M$ that satisfies \eqref{eq:I} for some $T>0$ lifts to a curve
  $x:\R\rightarrow T^* E$ with
\begin{equation}
\left \{\begin{array}{l}
          \dot x(t) = X_{H}(x(t))\,;\\
          x(T) = e^{i\vp} x(0)\, ;\\
          H(x)=\bar k+\frac 12\, ;\\
          A(x)=1\, ,
          \end{array}\right.
\label{eq:II}
\end{equation}
for some $\vp \in \R$.  Conversely, a curve $x:\R\rightarrow T^*E$
satisfying \eqref{eq:II} projects to a closed magnetic geodesic with
energy $\bar k$.
\end{lem}
\begin{proof}
Given any $(q,p) \in A^{-1}(1)$ and $v \in T_q E$. Set $(\bar q,\bar p) :=\Pi_1(q,p)$ and $\bar v:=d_q \tau v$. Splitting into horizontal and vertical components we conclude by \eqref{eq:Pic}
$$\<p,v\> = \<p,v^\hor\> + \<Z,v\>, \quad \<p,v^\hor\>= \<\bar p,\bar v\>.$$ 
Hence by definition of the dual norm
\[
|p| = \max_{|v|^2 =1} \<p,v\> = \max_{x\in [-1,1]} \max_{|v^\hor|=\sqrt{1-x^2}} \<p,v^\hor\> + x = \max_x \sqrt{1-x^2}|\bar p| + x\,.
\]
By maximization in the $x$ variable we verify $|p| =\sqrt{|\bar p|^2+1}$. This shows $H = \bar H \circ \Pi_1 + \frac 12$. The rest follows since by Proposition~\ref{prp:red} we have $\Pi_1^*\bar \omega_\sigma = i^*\omega$. 
\end{proof}

\vspace{2mm}

\subsection{A Rabinowitz-type action functional.}
\label{arabinowitz}
Lemma~\ref{lem:sympreduction} above shows that, in order to find closed magnetic geodesics
with energy $\bar k$, it suffices to look for geodesics in $T^*E$
with kinetic energy $\bar k +\frac 12$ that are closed up to
$S^1$-action and which lie on the level set $A^{-1}(1)$. For our
variational approach we reformulate~\eqref{eq:II} into a problem of
closed curves with period $1$. More precisely, if $(x,T,\vp)$ is a solution
of~\eqref{eq:II}, then the curve $y:[0,1]\to T^*E$ defined by $y(t):=e^{-it \vp} x(t T)$
satisfies
\begin{equation}
\left \{\begin{array}{l}
          \dot y(t) = -\vp X_A(y(t)) + T X_{H}(y(t))\,;\\
          y(1) =  y(0)\, ;\\
          H(y)=\bar k+\frac 12\, ;\\
          A(y)=1\, .
          \end{array}\right.
\label{eq:III}
\end{equation}
Conversely, every solution of~\eqref{eq:III} gives a solution
of~\eqref{eq:II} by reversing the rescaling. 
\begin{lem}\label{lem:A}
  Set $k:=\bar k+\frac 12$. A triple $(y,T,\vp)$
  satisfies~\eqref{eq:III} if and only if it is a critical point of the functional 
  $\A_k:C^\infty(S^1,T^*E) \times (0,+\infty) \times \R\to \R$ given by
  \begin{equation}
  \A_k(y,T,\vp) = \int_0^1 y^* \lambda \ - \int_0^1 \big (TH_k(y) - \vp A_1(y)\big )\, dt\,,
  \label{functionalak}
  \end{equation}
  where $\lambda$ is the Liouville $1$-form, $H_k(q,p):= H(q,p) - k$
  and $A_1(q,p) := A(q,p) -1$.
\end{lem}
\begin{proof}
  Let $s \mapsto u_s
  \in C^\infty(S^1,T^*E)$ be a differentiable curve with $u_0=y$ and
  $$\xi:=\ \frac{d}{ds} \Big |_{s=0}\, u_s.$$  
  Abbreviate the Hamiltonian $\widehat H :=
  TH_k-\vp A_1$ and use 
  $$\omega(\ps u,\pt u) = (d\lambda)(\ps u,\pt u) =
  \ps \lambda(\pt u) - \pt \lambda(\ps u)$$ 
  to conclude that
  \begin{align*}
    d\A_k(y)[\xi] &= \int_0^1 \omega(\xi,\dot y) - \int_0^1 d\widehat H(\xi)\, d t\\
    &=\int_0^1 \omega(\xi,\dot y) + \int_0^1 \omega(X_{\widehat H}(y),\xi)\, dt =\int_0^1\omega(\xi, \dot y - X_{\widehat H}(y))\, dt\,.
  \end{align*}
  
  If $(y,T,\vp)$ solves~\eqref{eq:III}, then clearly $d\A_k(y)[\xi]=0$ for all $\xi$. On the other hand, if 
  $d\A_k(y)[\xi]=0$ for all $\xi$, then by the fundamental lemma of calculus of variations and by non-degeneracy of $\omega$ 
  the curve $y$ has to solve the first equation in~\eqref{eq:III}. Differentiating $\A_k$ in direction $T$ and
  $\vp$ shows that 
  $$\frac{\partial \A_k}{\partial T} (y,T,\vp) = -\int_0^1 H_k(y), \qquad \frac{\partial \A_k}{\partial \vp} (y,T,\vp) = \int_0^1 A_1(y)\, .$$
  
  Now it is clear that $\partial \A_k/\partial T(y,T,\vp)=\partial
  \A_k/\partial \vp (y,T,\vp)=0$ if $(y,T,\vp)$ is a solution
  of~\eqref{eq:III}. On the other hand, if $(y,T,\vp)$ is a critical
  point of $\A_k$ then $H$ and $A$ are constant along $y$ and hence $H(y)=k$ and $A(y)=1$ as required.
\end{proof}

The functional $\A_k$ in \eqref{functionalak} can be thought of as the classical Rabinowitz action functional (c.f. \cite{Albers:2010vk,Rabinowitz:1978ts,Rabinowitz:1979}) with two Lagrange multipliers instead of only one and
fits precisely in the setting considered in \cite{Kang2013}, where Rabinowitz-Floer homology for contact coisotropic submanifolds is defined. Notice indeed that, in the setting of the lemma above, $\Sigma:=H^{-1}(k)\cap A^{-1}(1)$ is a 
coisotropic submanifold of $T^*E$ of codimension 2, for the Hamiltonians $H$ and $A$ Poisson-commute.
Therefore, it is not unreasonable to try to use Rabinowitz-Floer homology to infer existence results of critical points of the functional $\A_k$. However, this is very far from being a straightforward application of the results in 
\cite{Kang2013}. Indeed, the coisotropic submanifold $\Sigma$ is in general not of contact type (c.f. Section 3), even though all energy level sets of $H$ are trivially of contact type on $(T^*E,\omega)$.
Notice that the latter fact is in sharp contrast with what happens on $(T^*M,\bar \omega_\sigma)$, where very little is known about the contact property for energy level sets of the kinetic Hamiltonian.
In fact, low energy levels on surfaces different from the two-torus are known to be not of contact type, in case for instance $\sigma$ is an exact form (c.f. \cite[Theorem 1.1]{Contreras:2004lv}); it is however an open problem to determine whether such energy levels are stable or not. We refer to \cite{Hofer:1994bq} for the definition of stability and (for instance) to \cite[Corollary 8.4]{Abbondandolo:2013is} for the relation between the stability property and the existence of periodic orbits.
Anologously, one could ask whether the coisotropic submanifold $\Sigma$ is stable or not. This will be done in the next section. 

We finish this section noticing that we might not expect the existence of critical points of $\A_k$ for every $k$, as the example of the horocycle flow \cite{Hedlund:1932am} shows.



\section{Stability and contact property of coisotropic submanifolds}
\label{stabilityandcontactpropertyofcoisotropicsubmanifolds}

In the previous section we showed that, in order to prove the existence of solutions to  \eqref{eq:I}, it suffices to show the existence of 1-periodic orbits for the Hamiltonian flow defined by the Hamiltonian $T\cdot H - \varphi\cdot A:T^*E \rightarrow \R$, for some $T>0$, $\varphi \in \R$, and the standard symplectic form on $T^*E$ which are contained in the coisotropic 
submanifold $\Sigma:= H^{-1}(k)\cap A^{-1}(1)$ or, equivalently, to show the existence of critical points of the Rabinowitz-type action functional $\A_k$ given by \eqref{functionalak}.
In order to potentially apply the techniques developed in \cite{Kang2013} we first need to know that $\Sigma$ is of contact type or, at least, stable.

Let us first recall the notions of contact type, resp. stable coisotropic submanifold, which were introduced by Bolle in \cite{Bolle:1996,Bolle:1998}.
For examples of stable resp. contact type coisotropic submanifolds we refer to \cite{Kang2013}. Other examples in the setting considered in the present paper will be discussed in the next subsections.

\begin{dfn}
Let $(Y^{2m},\omega)$ be a symplectic manifold and let $H_0,...,H_{k-1}:Y\rightarrow \R$ be Poisson-commuting Hamiltonians such that zero is a regular value for each function and such that the intersection of the zero-energy level sets of $H_0,...,H_{k-1}$
$$\Sigma := \bigcap_{j=0}^{k-1} H^{-1}_j(0)$$
is cut-out transversely. Then $\Sigma$ is a $(2m-k)$-dimensional coisotropic submanifold. The coisotropic submanifold $\Sigma$ is called $\mathsf{stable}$ if there exist
one-forms $\alpha_0,...,\alpha_{k-1}$ such that $\ker \omega_\Sigma\subseteq \ker \alpha_j$, for all $j=0,...,k-1$, and
$$\alpha_0\wedge ... \wedge \alpha_{k-1}\wedge  \omega_\Sigma^{2(m-k)} \neq 0$$
everywhere on $\Sigma$, where $\omega_\Sigma$ denotes the restriction of $\omega$ to $\Sigma$. We say that $\Sigma$ is $\mathsf{of \ contact \ type}$ if the stabilizing forms 
$\alpha_0,...,\alpha_{k-1}$ can be chosen within the set of all primitives of $\omega_\Sigma$.
\end{dfn}

Obviously a necessary condition for $\Sigma$ to be contact is that the restricted symplectic form $\omega_\Sigma$ is exact. Furthermore, being of contact type for closed coisotropic submanifolds of codimension higher than one is also topologically obstructed. 
\begin{lem}
\label{obstruction}
If $\Sigma$ is contact, then $\dim H^1(\Sigma,\R) \geq k$. 
\end{lem}
\begin{proof}
Suppose by contradiction that $\Sigma$ is contact and $\dim H^1(\Sigma,\R) <k$. Let $\alpha_0,...,\alpha_{k-1}$ be primitives of $\omega_\Sigma$ satisfying the requirements of the definition 
above. By assumption the $k$ cohomology classes $[\alpha_0],...,[\alpha_{k-1}]$ are linearly dependent, that is there exists coefficients $\lambda_0,...,\lambda_{k-1}\in \R$ not all equal to zero such that 
$$\lambda_0 [\alpha_0]+...+\lambda_{k-1}[\alpha_{k-1}]=0.$$
Without loss of generality we assume that $\lambda_0=1$. The equation above means that 
$$\alpha_0 = df- \lambda_1 \alpha_1 - \dots - \lambda_{k-1} \alpha_{k-1}\,,$$
for some function $f:\Sigma\rightarrow \R$. Therefore 
\[
\alpha_0\wedge ... \wedge \alpha_{k-1}\wedge  \omega_\Sigma^{2(m-k)} = df \wedge \alpha_1 \wedge \dots \wedge \alpha_{k-1}\wedge \omega_\Sigma^{2(m-k)}\,,\]
which vanishes at every critical point of $f$. Since $\Sigma$ is closed, $f$ has at least a critical point and hence we conclude that the form 
$$\alpha_0\wedge ... \wedge \alpha_{k-1}\wedge  \omega_\Sigma^{2(m-k)}$$
is never a volume form on $\Sigma.$
\end{proof}
%

\subsection{Coisotropic submanifolds arising via symplectic reduction.}
In the case we are interested in, i.e. when $\Sigma =H^{-1}(k)\cap A^{-1}(1)$ is a codimension two coisotropic submanifold of $T^*E$, we have that $\Sigma$ is stable if
there exist one-forms $\alpha_0,\alpha_1$ on $\Sigma$ such that $\ker \omega_\Sigma \subseteq \ker d\alpha_i$,
for $i=0,1$, and 
\begin{equation}
\alpha_0\wedge \alpha_1 \wedge \omega_\Sigma^{n} \neq 0.
\label{stability}
\end{equation}
Here $n$ denotes the dimension of $M$. Recall that $\Sigma$ is said of \textit{contact type} if the stabilizing one-forms $\alpha_0$ and $\alpha_1$
are primitives of $\omega_\Sigma$ and the obstruction to the contact type condition as discussed in Lemma \ref{obstruction} reads: $H^1(\Sigma,\R)\neq 0$.

In what follows we write $\Sigma=\Sigma_0\cap \Sigma_1$, where 
$$\Sigma_0 := H_k^{-1}(0), \qquad \Sigma_1:= A_1^{-1}(0),$$
with $H_k$ and $A_1$ are as in the statement of Lemma \ref{lem:A}, and denote with $X_0$ and $X_1$ the Hamiltonian vector fields of $H_k$ and $A_1$ respectively. The following lemma provides a criterion for the contact property of $\Sigma$ in terms of the Hamiltonian vector fields $X_0$ and $X_1$. A similar statement holds clearly also for the stability condition.

\begin{lem}
The following facts are equivalent:
\begin{enumerate}[1.]
\item $\Sigma$ is of contact type.
\item There exist primitives $\alpha_0,\alpha_1$ of $\omega_\Sigma$ such that the following matrix in non-singular on $\Sigma$:
\begin{equation}\label{eq:matrix}
\left (\begin{matrix} \alpha_0(X_0) & \alpha_0(X_1) \\ \alpha_1(X_0) & \alpha_1(X_1)\end{matrix}\right ).
\end{equation}
\end{enumerate}
\label{lemcondition}
\end{lem}

\begin{proof}
The two-form $\omega_\Sigma$ has kernel on $\Sigma$ generated exactly by the Hamiltonian vector fields $X_0$ and $X_1$. In particular, the matrix in $(2)$ is non-singular everywhere on $\Sigma$ if and 
only if the contraction of the form in \eqref{stability} by $X_0$ and $X_1$ is non-zero on the complement of $\ker \omega_\Sigma$.
\end{proof}

In what follows we denote by $\bar \Sigma := \bar H^{-1}(\bar k) =\Sigma /S^1$ the quotient of $\Sigma$ with respect to the $S^1$-action on $T^*E$. 

\begin{lem}
The following statements hold: 
\begin{enumerate}[1.]
\item If $\bar \Sigma$ is of contact type in $(T^*M,\bar \omega_\sigma)$, then $\Sigma$ is of contact type in $(T^*E,\omega)$.
\item\label{nm:stable} The hypersurface $\bar \Sigma$ is stable in $(T^*M,\bar \omega_\sigma)$ if and only if $\Sigma$ is stable in $(T^*E,\omega)$.
\end{enumerate}
\label{lem:correspondence}
\end{lem}

\begin{proof}
\begin{enumerate}[1.]
\item Let $\bar \alpha$ be a contact form for $\bar \Sigma$ and consider $\alpha_1:=\pi^*\bar \alpha$, $\alpha_0=\lambda_\Sigma$ restriction to $\Sigma$ of the Liouville 1-form on $T^*E$. By definition we have 
$$\omega_\Sigma=\pi^*\bar \omega_{\sigma}|_{\bar \Sigma} = \pi^* d\bar \alpha = d\alpha_1=d\alpha_0.$$
By contruction we have $d\pi X_0 = \bar X$, $d\pi X_1=0$, where $\bar X$ denotes the Hamiltonian vector field defined by the kinetic Hamiltonian and the twisted symplectic form on $T^*M$. It follows by the 
contact condition that
$$\alpha_0(X_1)\equiv 1, \quad \alpha_1(X_0)= \bar \alpha (\bar X)\neq 0,\quad \alpha_1(X_1) =\bar \alpha (0)=0,$$
and hence the matrix in Lemma \ref{lemcondition} is nowhere singular on $\Sigma$.
\vspace{2mm}
\item Suppose now that $\bar \Sigma$ is stable with stabilizing form $\bar \alpha$ and consider the one-forms $\alpha_0,\alpha_1$ on $\Sigma$ as above. It suffices to show that $\ker \omega_\Sigma\subseteq \ker d\alpha_1$. 
By the stability property of $\bar \Sigma$ we know that any vector $v\in \ker \omega_\Sigma$ projects to a vector in $\ker d\bar \alpha$, since $\bar v := d\pi v \in \ker \bar \omega_\sigma|_{\bar \Sigma}\subseteq \ker d\bar \alpha.$
It follows that for all $w\in T\Sigma$ we have
$$(d\alpha_1)(v,w) = d\pi^*\bar \alpha (v,w) = \pi^*d\bar \alpha(v,w) = d\bar \alpha(\bar v, \bar w) =0$$
and hence $v\in \ker d\alpha_1$. Conversely, suppose that $\Sigma$ is stable and let $\beta_0,\beta_1$ be a stabilizing pair for $\Sigma$. Starting from $\beta_0,\beta_1$ we define a new stabilizing pair $\beta_0',\beta_1'$ 
for $\Sigma$ which is invariant under the flow of $X_1$ (denoted by $\phi_1^t$) by
$$\beta_i'(v) := \int_0^1 (\phi_1^t)^*\beta_i[v]\, dt, \quad \forall v\in T_p\Sigma, \ p\in \Sigma, \ i=0,1.$$
Since 
$$d\beta_i' = \int_0^1 (\phi_1^t)^*d\beta_i \, dt$$
and $\phi_1$ preserves $\ker \omega_\Sigma$ (since it preserves $\omega_\Sigma$), we have that $\ker \omega_\Sigma \subseteq \ker d\beta_i'$, for $i=0,1$. Moreover, since by assumption $\beta_0\wedge \beta_1\wedge \omega_\Sigma^{n-2}\neq 0$, we can conclude that $\beta_0'\wedge \beta_1'\wedge \omega_\Sigma^{n-2}\neq 0$. By construction we have $(\phi_1^t)^*\beta_i'=\beta_i'$ for all $t\in \R$, for $i=0,1$. Deriving in $t$ and 
evaluating at $t=0$ yields 
\begin{equation}\label{eq:magic}
0=\frac{d}{dt} (\phi_1^t)^*\beta_i' \Big |_{t=0} = \mathcal L_{X_1} \beta_i' = d(\imath_{X_1}\beta_i') + \imath_{X_1} d\beta_i' =d(\imath_{X_1}\beta_i').
\end{equation}
This shows that the functions $\beta_0'(X_1)$ and $\beta_1'(X_1)$ are constant along $\Sigma$. We set $b_0:=\beta_0'(X_1)$, $b_1:=\beta_1'(X_1)$, and denote by $\Pi:\Sigma \to \bar \Sigma$ the quotient map. Finally, we define a 1-form $\bar \beta$ implicitly via $\Pi^*\bar \beta = b_1 \beta_0' - b_0\beta_1'$, i.e.
$$\bar \beta_{\bar p}(\bar v):= b_1\cdot (\beta_0')_p(v) - b_0\cdot (\beta_1')_p(v),$$
for all $p\in \Sigma$ in the fibre over $\bar p$ and $v\in T_p\Sigma$ such that $d_p \Pi v= \bar v$. Notice that this is a good definition since $\beta_0'$ and $\beta_1'$ are $\phi_1^t$-invariant and by construction the right-hand side  vanishes on the kernel of $d\Pi$, which is spanned by the vector field $X_1$. Since $d\Pi X_0=\bar X$ we conclude 
\[\bar \beta(\bar X) = b_1 \beta'_0(X_0) - b_0 \beta'_1(X_0)=\det (\beta'_i(X_j))\neq 0\,,\]
which implies that that $\ker \bar \omega_\sigma |_{\bar \Sigma}\subseteq \ker d\bar \beta$. \qedhere
\end{enumerate}
\end{proof}

\begin{rem}
The contact condition for $\Sigma$ is in general weaker than the contact condition for $\bar \Sigma$ as the following example shows. 
Consider the flat torus $(\T^2,g)$ and let $\sigma$ be the area form induced by $g$. Then energy levels $\bar H^{-1}(\bar k)$ are stable in $(T^*\T^2, \bar \omega_\sigma)$ for every $\bar k>0$ with stabilizing 
form given by the angular form $d\theta$ but never of contact type, for the 2-form $\pi^*\sigma|_{\bar H^{-1}(\bar k)}$ is never exact (in fact, the map $\pi^*:H^2(\T^2)\rightarrow H^2(\bar H^{-1}(\bar k))$ is injective). 
However, the associated coisotropic submanifold $\Sigma$ in $T^*E$ is of contact type with contact forms given by $\alpha_0$ and $\alpha_1:=\alpha_0+\tau^*d\theta$, where $\alpha_0$ denotes the restriction of the Liouville 1-form 
to $\Sigma$.

Arguing as in the proof of Statement~\ref{nm:stable} in Lemma \ref{lem:correspondence} we see that $\bar \Sigma$ is of contact type in $(T^*M,\bar \omega_\sigma)$, provided that $\Sigma$ is of contact type in $(T^*E,\omega)$ and 
the constants $b_0,b_1$ satisfy $b_0+b_1=1$.
\end{rem}


\subsection{Examples.} 
From Lemma \ref{lem:correspondence} we deduce that all examples of stable resp. contact type hypersurfaces in $(T^*M,\bar\omega_\sigma)$ discussed in \cite{Cieliebak:2010zt} give rise to examples of stable, resp. contact type coisotropic submanifolds in $(T^*E,\omega)$. From \cite{Cieliebak:2010zt} we also get examples of non-stable coisotropic submanifolds.  We now explain another class of examples arising from homogeneous spaces. 

Let $G$ be a compact Lie group and $H \subset G$ a closed subgroup. Fix an $\Ad$-invariant inner product $\<\cdot,\cdot\>$ on the Lie algebra of $G$ and define a metric on $G$ via $\<v,w\>_q =\<dL_q^{-1} v,dL_q^{-1} w\>$ for all $q \in G$ and $v,w \in T_qG$, where $L_q:G\to G$, $g\mapsto qg$ denotes the left-multiplication. The group $H$ acts on $G$ by right-multiplication and we define a metric on the quotient $M:= G/H$ by requiring that the canonical projection $G \to M$ is a Riemannian submersion. Assume that there exists a closed normal  subgroup $H_0 \subset H$ such that $H/H_0 \cong S^1$ or equivalently that there exists a group homomorphism $H\cong H_0\times S^1$. As above, we obtain a metric on the quotient $E:=G/H_0$. Clearly the residual action of $H/H_0 \cong S^1$ descends to the quotient $E$ and thus $E$ is a circle bundle with base $M$. Let $Z$ denote the corresponding fundamental vector field on $E$. We assume without loss of generality that  $|Z|=1$ and observe that the metric on the $S^1$-bundle $E$ constructed in that way is of the form considered in Section~\ref{s:symplecticreduction}. As explained there, we obtain a connection form $\theta\in \Omega^1(E)$, the corresponding the curvature form $\sigma\in\Omega^2(M)$ as well as the twisted symplectic form $\bar \omega_\sigma=\bar \omega + \bar \pi^*\sigma$ for $T^*M$. 

To infer stability, in the next lemma we will need an additional regularity assumption that we now explain, even though we believe that this assumption could be dropped. Let $\g$, $\h$ and $\h_0$ denote the Lie algebras of the Lie groups $G$, $H$ and $H_0$ respectively. The embeddings $H_0 \subset H \subset G$ induce the inclusions $\h_0 \subset \h \subset \g$. Let $\zeta \in \h$ be such that $|\zeta|=1$ and $\zeta \perp \h_0$, i.e.\ $\zeta$ is orthogonal to any element of $\h_0$. Recall that an element $p \in \g$ is called \emph{regular} if the adjoint map $\ad_p=[p,\cdot]:\g \to \g$ has maximal rank, i.e.\ $\rk \ad_p =\max_{p' \in \g} \rk \ad_{p'}$.  

\begin{lem}
Assume that, for a given $\bar k>0$, $S_{\bar k}:= \{\bar p + \zeta \in \g \ |\ \bar p \perp \h, |\bar p|^2 = 2\bar k\}$ 
contains only regular elements. Then 
$$\bar \Sigma_{\bar k} := \Big \{(\bar q,\bar p) \in T^*M \ \Big |\  |\bar p|^2 =2\bar k\Big \}$$
is stable in $(T^*M,\bar \omega_\sigma)$.
\end{lem}
\begin{proof}
By Lemma~\ref{lem:correspondence} it suffices to show that
 $$\Sigma_k:=\Big \{ (q,p) \in T^*E \ \Big |\ |p|^2 =2k, \ p(Z) = 1\Big \}$$ 
 is stable in $(T^*E,\omega)$, where $k=\bar k +\frac 12$. 
As first stabilizing form we choose $\alpha_0$, the restriction of the Liouville form in $T^*E$ to $\Sigma_k$. The definition of $\alpha_1$ requires instead a little preparation. 
The splitting $\g=\h_0 \oplus \h_0^\perp$ is preserved under the adjoint action of $H_0$. We have the well-known isomorphism 
\[
\phi:G \times_{H_0} \h_0^\perp \stackrel{
\cong}{\longrightarrow} TE,\qquad [g,v] \mapsto d\tau_0 dL_g v\,,
\]
where $\tau_0:G \to E$ denotes the quotient map, $L_g:G \to G$ the left-multiplication with the element $g$ and $G \times_{H_0} \h_0^\perp$ the associated bundle, for which $H_0$ acts on $G$ by right-multiplication and on $\h_0^\perp$ with the adjoint action. We identify $T^*E$ with $TE$ using the metric and see that by construction we have $
\phi^{-1}(\Sigma_k) = G \times_{H_0} S_k$. 

For each $(q,p) \in TE$ consider the splitting into horizontal and vertical space with respect to the Levi-Civita connection
\[T_{(q,p)} TE = T^\hor_{(q,p)} TE \oplus T^\mathrm{ver}_{(q,p)} TE\,.\]
Recall that both factors are canonically isomorphic to $T_q E$. 
For each $p \in \g$ denote by $\g_p := \ker \ad_p \subset \g$ and by $\pi_{\g_p}$, $\pi_{\g^\perp_p}$ the orthogonal projections to the respective subspaces. By construction $\ad_p$ is invertible on $\g^\perp_p$. We finally define $\alpha_1$
\[
(\alpha_1)_{[q,p]}(\phi([q,Q])^\hor+\phi([q,P])^\mathrm{ver}) = \<\pi_{\g_p} \zeta,Q\>+ \<\ad_p^{-1} \pi_{{\g_p}^\perp} \zeta,P\>\,,
\]
for all $Q,P \in \h_0^\perp$ and $[q,p] \in G \times_{H_0} S_k$, where $\cdot^\hor$ and $\cdot^\mathrm{ver}$ denotes the horizontal and vertical lift respectively. One verifies directly that the right-hand side has the correct invariance property and depends smoothly on $p \in S_k$, since by assumption $S_k$ only contains regular elements. By abuse of  notation we abbreviate by $(Q,P)$ the argument of $
\alpha_1$ in the followig discussion. 

We now check that the pair $(\alpha_0,\alpha_1)$ indeed defines a stable pair. Let $X_0$ and $X_1$ be the Hamiltonian vector fields of $H_k$ and $A_1$ respectively. We immediately verify that $\alpha_0(X_0)=|p|^2 = 2k$ and $\alpha_0(X_1)= \<p,Z\>=1$.  In the splitting, $X_0$ at $[q,p]$ is given by $(p,0)$ and we have
\begin{equation}\label{eq:alpha1}
(\alpha_1)_{[q,p]}(p,0) = \<\pi_{\g_p} \zeta,p\> = \<\zeta,p\>=|\zeta|^2 = 1\,.
\end{equation}
By definition $X_1$ is the generator of the lifted $S^1$-action on $T^*E \cong TE$ given by \emph{right}-multiplication with $t\mapsto \exp(t \zeta)$. So, in the splitting, $X_1$ at $[q,p]$ is given by $(\zeta, \ad_\zeta p)$. 
We compute
\begin{align*}
(\alpha_1)_{(q,\phi_q(p))}(\zeta,\ad_\zeta p) &= \<\pi_{\g_p} \zeta,\zeta\> + \<\ad_p^{-1} \pi_{g^\perp_p} \zeta, \ad_\zeta p\>= \<\pi_{\g_p} \zeta, \zeta\> + \< \pi_{\g^\perp_p} \zeta, \zeta\>= 1\,.
\end{align*}
We conclude that 
\[
\det \begin{pmatrix}
\alpha_0(X_0) &\alpha_0(X_1)\\
\alpha_1(X_0)&\alpha_1(X_1)
\end{pmatrix} = \det \begin{pmatrix}
2k&1\\ 
1&1
\end{pmatrix} = 2k-1 >0\,.
\]
It remains to check that $\ker \omega_{\Sigma} \subset \ker d\alpha_i$  or equivalently that $X_0,X_1 \in \ker d\alpha_1$ for $i=0,1$. Since $\alpha_0$ is a primitive of $\omega_\Sigma$ there is nothing to show for $i=0$. The right-hand side of~\eqref{eq:alpha1} is invariant not only under the adjoint action of $H_0$ but also under the action of $H$, for the one form $\alpha_1$ is invariant under the $S^1$-action. Because $\alpha_1(X_1)=1$ is constant and $\alpha_1$ is invariant, we obtain for free $(d\alpha_1)(X_1,\cdot)=0$ by a similar computation as in~\eqref{eq:magic}. We compute
\[
(d\alpha_1)_{[q,p]} ((Q_1,0),(Q_2,0))  = \<\pi_{\g_p} \zeta,\ad_{Q_1}Q_2\>\,,
\]
for any $Q_1,Q_2 \in \h_0^\perp$. In particular  since by definition $\pi_{\g_p} \ad_p Q_2 =0$ we have 
\[(d\alpha_1)_{[q,p]}((p,0),(Q_2,0)) = \<\pi_{\g_p} \zeta,\ad_p Q_2\> = 0\,.\qedhere
\]
\end{proof}


\section{The Lagrangian action functional $\S_k$}
\label{thelagrangianactionfunctional}

Unfortunately, the functional $\A_k$ defined in \eqref{functionalak} is not well-suited for finding critical points using classical Morse theory. In fact, the natural space over which it is defined - namely $H^{1/2}(S^1,T^*E)$ - does not have a good structure of an infinite dimensional manifold due to the fact that curves of class $H^{1/2}$ might have discontinuities.
Furthermore, the functional $\A_k$ turns out to be strongly indefinite, meaning that all its critical points have infinite Morse index and coindex.
Therefore, using the Legendre transform $\mathcal L:TE\to T^*E$, we introduce a related Lagrangian action functional $\S_k$ defined on the product Hilbert manifold $H^1(S^1,E)\times (0,+\infty)\times \R$,
whose critical points correspond to those of $\A_k$. Here $H^1(S^1,E)$ denotes the space of absolutely continuous loops $\gamma:S^1\rightarrow E$ with
square-integrable first derivative; it is well-known that $H^1(S^1,E)$ has a natural structure of Hilbert manifold (c.f. \cite{Abbondandolo:2009gg}) with Riemannian metric 
$g_{H^1}$ naturally induced by the metric $g^\alpha$. On $\mathcal M:= H^1(S^1,E)\times (0,+\infty)\times \R$ we then consider the product metric $g_{\mathcal M}=
g_{H^1}+dT^2+d\vp^2$. Observe that $(\mathcal M,g_{\mathcal M})$ is not complete. 

In the following we will prove the existence of critical points of $\S_k$ using variational methods, even though the functional $\S_k$ might fail to satisfy a crucial compactness property 
(namely the \textit{Palais-Smale condition}). To overcome this difficulty we will use a monotonicity argument, better known 
as the \textit{Struwe monotonicity argument}, which is originally due to Struwe \cite{Struwe:1990sd} and has been already successfully applied \cite{Abbondandolo:2013is,Asselle:2015ij,Asselle:2015sp,Asselle:2014hc,Asselle:2016qv,Contreras:2006yo} to the existence of closed magnetic geodesics.

We recall that the connected components of $\mathcal M$ are in one to one correspondence 
with the set of conjugacy classes in $\pi_1(E)$, for the canonical inclusions 
$$C^\infty (S^1,E)\hookrightarrow H^1(S^1,E)\hookrightarrow C^0(S^1,E)$$ 
are dense homotopy equivalences. Finally, we denote with $\mathcal M_0$ the connected component of $\mathcal M$ given by the contractible loops.


\subsection{The variational principle.} As in the previous sections we denote with $Z$ the fundamental vector field of the $S^1$-action on $E$. For fixed values of $T$ and
$\vp$ the Legendre transform $\mathcal L:TE\to T^*E$ of the Tonelli Hamiltonian $\widehat H:= TH-\varphi A$ yields the following Tonelli Lagrangian 
$$L_{T,\vp}:TE\to \R,\qquad L_{T,\vp}(q,v) = \frac{1}{2T}|v+\vp Z(q)|^2-\vp +kT,$$
where $k:=\bar k+\frac 12$, and an associated Lagrangian action functional $H^1(S^1,E)\to \R$,
$$\gamma \longmapsto \frac{1}{2T}\int_0^1 |\dot \gamma(t)+\vp Z(\gamma(t))|^2\, dt - \vp +kT.$$
By letting the values of $T$ and $\vp$ free we thus get a functional $\S_k:\mathcal M\rightarrow \R$,
\begin{equation}
\S_k(\gamma,T,\vp)= \frac{1}{2T}\int_0^1 |\dot \gamma(t)+\vp Z(\gamma(t))|^2 \, dt - \vp +kT.
\label{functionalsk}
\end{equation}

For sake of completeness we now verify that critical points of $\S_k$ project to $T$-periodic magnetic geodesics with energy $\bar k$.
In order to do that we need an auxiliary lemma. In what follows we denote with $\<\cdot,\cdot\>$ the metric $g^\alpha$ on $E$ as constructed in 
Section \ref{s:symplecticreduction} and with $\na$ the associated Levi-Civita connection.
\begin{lem}\label{lem:dalpha}
  For all $u,v\in TE$ we have 
  $$d\alpha(u,v) = 2\<\na_u Z,v\>.$$
\end{lem}
\begin{proof}
  We denote by $\Phi$ the flow of $Z$. Consider
  $c(s,t)=\Phi^s\gamma(t)$ for some path $\gamma$ in $E$ with $\pt
  \gamma(0)=u$. Since by construction $\Phi^s$ is an
  isometry for each $s$, we have
  $$|\pt c(s,t)| =|d\Phi^s \pt \gamma(t)|=|\pt  \gamma(t)|,\quad \forall s\in \R. $$
  In particular
  \[0 = \frac 12 \ps |\pt c|^2 = \<\na_s \pt c,\pt c\> = \<\na_t \ps c,\pt c\>=\<\na_t Z,\pt c\>\,.\] 
  Thus $\<\na_u Z,u\>=0$, for all $u$. This
  shows that the tensor $(u,v) \mapsto \<\na_u Z,v\>$ is
  skewsymmetric.  Now let $(s,t) \mapsto c(s,t)$ be any map such
  that $\ps c(0)=u$ and $\pt c(0)=v$. We have $\alpha(\ps c)= \<\ps
  c,Z\>$. Deriving by $\pt$ gives
  \[
  \pt \alpha(\ps c) = \<\na_t \ps c,Z\> + \<\ps c,\na_t Z\>\,.
  \]
  Interchanging the role of $\ps$ and $\pt$ gives 
  $$\ps \alpha(\pt c)=\<\na_s\pt c,Z\>+ \<\pt c,\na_s Z\>.$$
  Finally, subtracting the two equations we get by skewsymmetry
  \begin{align*}
  d\alpha(\ps c,\pt c) &= \ps \alpha(\pt c) - \pt \alpha(\ps c)\\
                                 &= \<\na_s\pt c,Z\>+ \<\pt c,\na_s Z\> - \<\na_t \ps c,Z\> - \<\ps c,\na_t Z\>\\
                                 &= 2 \<\pt c,\na_s Z\>.\qedhere
  \end{align*}
\end{proof}

\begin{prop}
  If $(\gamma,T,\vp)\in \mathcal M$ is a critical point of $\S_k$ then the periodic curve $\mu:[0,T]\to M$ defined  by
  $\mu(t):= \tau\circ \gamma(t/T)$ for all $t\in [0,T]$ is a closed magnetic geodesic with energy $\bar k$. Conversely, for every $T$-periodic 
  closed magnetic geodesic $\mu:[0,T]\to M$ with energy $\bar k$ there exist $\gamma:S^1\to E$ and $\vp\in \R$ such that $\mu(\cdot)=\tau\circ \gamma(\cdot/T)$ and 
  $(\gamma,T,\vp)\in \mathcal M$ is a critical point of $\S_k$.
\end{prop}
\begin{proof}
  Consider a variation $s \mapsto \gamma_s \in H^1(S^1,E)$ with
  $\gamma:=\gamma_0$ and
  $$\xi:=\frac{d}{ds}\Big |_{s=0}\, \gamma_s.$$  
  Differentiating $\S_k$ in the $\gamma$-variable and evaluating at
  $\xi$ yields
  \begin{align*}
    d_\gamma \S_k(\gamma)[\xi] &= \frac 1 T \int_0^1\<\dot \gamma + \vp Z, \na_t \xi + \vp \na_\xi Z\>\, dt    \\
    & = \frac 1T \int_0^1 \Big [\<\dot \gamma,\na_t \xi\> +\vp \<Z,\na_t\xi\>+\vp\<\dot \gamma,\na_\xi Z\>+\vp^2\<Z,\na_\xi Z\>\Big ] dt\\
    &= \frac 1T \int_0^1 \Big [-\< \na_t \dot \gamma,\xi\> + \vp \big
    (\<\dot
    \gamma,\na_\xi Z\> -\< \na_t Z,\xi\>\big )\Big ] dt\\
    &= \frac 1T\int_0^1 \Big [-\< \na_t \dot \gamma,\xi\> + 2 \vp \<\dot \gamma,\na_\xi Z\>\Big ] dt\\
    &=\frac 1T \int_0^1 \Big [-\<\na_t \dot \gamma,\xi\> + \vp
    d\alpha(\xi,\dot \gamma)\Big ] dt,
  \end{align*}
  where in the third equality we have used integration by parts and the fact that 
  $$\<Z,\na_\xi Z\>=\frac 12 \frac{d}{ds}\Big|_{s=0} |Z|^2 =0,$$ 
  in the penultimate one the skewsymmetry of the tensor $\<\na_u
  Z,v\>$ and in the last one Lemma \ref{lem:dalpha}. As
  $(\gamma,T,\vp)$ is a critical point of $\S_k$, the above quantity
  has to vanish for every choice of $\xi$ and hence we conclude that
  \begin{equation}
    \label{eq:dotgamma}
  \<\na_t \dot \gamma,\cdot\> =  \vp d\alpha(\cdot,\dot \gamma).  
  \end{equation}
  Differentiating $\S_k$ in the $T$-direction yields
   \begin{equation}
   0=-\frac 1{2T^2} \int_0^1 |\dot \gamma + \vp Z|^2\,dt +k\,, \quad
   \Rightarrow \quad \int_0^1 |\dot \gamma+ \vp Z|^2\, dt =2T^2k,
   \label{dskdT}
  \end{equation}
  whilst differentiating $\S_k$ in the $\vp$-direction gives
  \begin{equation}
  0= \frac 1T \int_0^1 \<\dot \gamma+ \vp Z,Z\>\, dt -1\,, \quad
  \Rightarrow \quad \int_0^1 \<\dot \gamma,Z\>\, d t = T -\vp.
  \label{dskdvp}
  \end{equation}
  Now observe that by Lemma \ref{lem:dalpha} and \eqref{eq:dotgamma}
  $$\frac{d}{dt}\<\dot \gamma,Z\> = \<\na_t \dot \gamma,Z\> +\<\dot \gamma,\na_t Z\>=0;$$ 
 therefore $\<\dot \gamma,Z\>$ is constant and hence by \eqref{dskdvp} we have 
 \begin{equation}
 \<\dot \gamma,Z\>=T-\vp.
 \label{dskdvp2}
 \end{equation}
 Similarily using Lemma \ref{lem:dalpha} and equation
 \eqref{eq:dotgamma} we conclude that
 \begin{align*}
   \frac{d}{dt} \<\dot \gamma +\vp Z,\dot \gamma +\vp Z\> &= 2\<\na_t
   \dot \gamma,\dot \gamma\> + 2\vp \<\na_t\dot \gamma,Z\> +
   2\vp\<\dot \gamma,\na_t Z\> = 0\,.
					\end{align*}
 This together with \eqref{dskdT} shows that  
 \begin{equation}
 |\dot \gamma+\vp Z|^2=2T^2k.
 \label{dskdT2}
 \end{equation}
 Now set $\mu(t):=\tau(\gamma(t/T))$ and use the splitting
  $$\dot \gamma = \xi + \<\dot \gamma,Z\> Z = \xi + (T-\vp) Z,$$ 
  with $\xi \in \ker  \alpha$. By construction we have $d\tau( \dot \gamma)=d \tau (\xi) = T \dot \mu $; therefore   
  $$2T^2 k = |\dot \gamma+ \vp Z|^2 =|\xi + TZ|^2 = T^2 |\dot \mu|^2+T^2$$
  and hence $\frac 12 |\dot \mu|^2 = \bar k = k -\frac 12$. Now, by definition we 
  have $\na_t \dot \gamma = \na_{\dot \gamma}\dot \gamma$; inserting the splitting $\dot \gamma=\xi + (T-\vp)Z$ in both
  arguments yields by \eqref{eq:dotgamma}
  \[
  \vp d\alpha(u,\dot \gamma)=\<\na_t \dot \gamma,u\>= T^2\<\na_{\dot
    \mu}\dot \mu,\bar u\> + 2(T-\vp)\<\na_\xi Z,u\>\,,
  \]
  where $u \in TE$ and $\bar u =d\tau (u)$. Since $\tau^*\sigma=d\alpha$ by  Lemma~\ref{lem:dalpha} we get
  \[
  \vp T \sigma(\bar u,\dot \mu) = T^2\<\na_t \dot \mu,\bar u\> + T(\vp -T)\sigma(\bar u,\dot \mu),
  \]
  which implies that 
  $$\<\na_t \dot \mu, \bar u\>=\sigma(\bar u,\dot \mu)=\<\bar u, Y_\mu(\dot \mu)\>\  \Rightarrow \ \na_t \dot \mu = Y_\mu(\dot \mu),$$ 
  as we wished to prove.
\end{proof}

We shall notice that the last proposition does not imply that critical points of
$\S_k$ are in bijection with closed magnetic geodesics. In fact, there
are different critical points of $\S_k$, which project to the same
closed magnetic geodesic. Suppose indeed that $(\delta,T)$ is a closed magnetic geodesic with energy $\bar k$ and denote with 
$$\text{Crit} (\delta,T)\subseteq \mathcal M,\quad \text{Crit} (S^1\cdot (\delta,T))\subseteq \mathcal M,$$ 
the set of critical point of $\S_k$ ``above'' $(\delta,T)$ and above the critical circle $S^1\cdot (\delta,T)$ respectively. Here $S^1$ is the action on the loop space of $M$ given by translation of the base point. Morover, let $\mathcal H:=(S^1\times S^1)\rtimes \Z$
be the Heisenberg-type group, whose group operation $\oplus$ is defined by 
$$(r_0,s_0,u_0)\oplus (r_1,s_1,u_1) = (r_0+r_1+u_1s_0, s_0+s_1,u_0+u_1).$$
Then, $\mathcal H$ acts on $\mathcal M$ by
$$(r,s,u)\cdot (\gamma(t),T,\varphi)= \Big (e^{2\pi i(ut + r)}\cdot \gamma(t+ s),T, \varphi-2\pi u\Big ).$$
The functional $\S_k$ transforms as follows under the action:
$$\S_k((r,s,u)\cdot (\gamma,T,\varphi))= \S_k(\gamma,T,\varphi) + 2\pi u.$$
Finally, identify $S^1\times \Z = \{(r,0,u)\in \mathcal H\}$ and consider the associated sub-action on $\mathcal M$. If $(\gamma,T,\varphi)$ is a critical point 
of $\S_k$ that projects to $(\delta,T)$, then $\text{Crit} (\delta,T)$ is the orbit of $(\gamma,T,\varphi)$ und the $S^1\times \Z$-action, whilst $\text{Crit} (S^1\cdot (\delta,T))$ 
is the orbit of $(\gamma,T,\varphi)$ under the $\mathcal H$-action. 


\subsection{The Palais-Smale condition for $\S_k$.} As already explained in the introduction to this section we will prove the existence of critical 
points for $\S_k$ using variational methods. To this purpose we will need the following definition.

\begin{dfn}
A sequence $(\gamma_h,T_h,\vp_h)$ contained in a given connected component of $\mathcal M$ is
called a \textbf{Palais-Smale sequence at level c for $\S_k$} if 
$$\lim_{h\rightarrow +\infty}\ \S_k(\gamma_h,T_h,\vp_h)=c\, , \qquad \lim_{h\rightarrow +\infty}\ |d\S_k(\gamma_h,T_h,\vp_h)| =0\, .$$
\end{dfn}

In the definition above, $|\cdot|$ denotes, with slight abuse of notation, the (dual) norm on $T^*\mathcal M$ induced by the Riemannian metric $g_{\mathcal M}$.
Observe that a limit point of a Palais-Smale sequence for $\S_k$ is trivially a critical point of $\S_k$. Therefore, we need to look for necessary and 
sufficient conditions for a Palais-Smale sequence to admit converging subsequences. Before doing that we need a lemma comparing the behavior of $T_h$ and $\vp_h$
on a Palais-Smale sequence. In the following  we will denote with $e(\gamma):=\int_0^1 |\dot \gamma|^2\, dt$ the kinetic energy of a loop $\gamma:S^1\to E$. 
\begin{lem}
Suppose $(\gamma_h,T_h,\vp_h)$ is a Palais-Smale sequence for $\S_k$ at level $c$, then:
\begin{enumerate}
\item $T_h\rightarrow 0$ if and only if $\vp_h\rightarrow -c$.
\item The $T_h$'s are uniformly bounded from above if and only if the $\vp_h$'s are uniformly bounded. 
\item $T_h\rightarrow +\infty$ if and only if $\vp_h\rightarrow +\infty$.
\end{enumerate}
\label{lem:Tvp}
\end{lem}
\begin{proof}
  If $(\gamma_h,T_h,\vp_h)$ is a Palais-Smale sequence, then we have 
\begin{align}
c + o(1)& = \ \S_k(\gamma_h,T_h,\vp_h) \ =\  \frac{1}{2T_h}\int_0^1 |\dot \gamma_h+\vp_hZ(\gamma_h)|^2\, dt - \vp_h + kT_h\, ;\label{1}\\ 
o(1) & = \ \frac{\partial \S_k}{\partial T}(\gamma_h,T_h,\vp_h) \ =\  k -\frac{1}{2T_h^2}\int_0^1 |\dot \gamma_h+\vp_h Z(\gamma_h)|^2\, dt\, ;\label{2}\\
o(1) & = \ \frac{\partial \S_k}{\partial \vp}(\gamma_h,T_h,\vp_h) \ =\  \frac{1}{T_h}\int_0^1 \langle \dot \gamma_h,Z(\gamma_h)\rangle \, dt + \frac{\vp_h}{T_h} - 1\, .\label{3}
\end{align}
From \eqref{2} it follows that 
$$\frac{1}{2T_h} \int_0^1 |\dot \gamma_h+\vp_hZ(\gamma_h)|^2 \, dt \ =\ kT_h + T_ho(1)$$
and then by replacing in \eqref{1} we get 
$$kT_h + T_ho(1) -\vp_h + kT_h = c+o(1)$$
from which it follows that 
$$\vp_h = 2k T_h + T_h o(1) - c + o(1)\, .$$
This shows at once (1),(2) and (3).
\end{proof}

\begin{lem}
Suppose $(\gamma_h,T_h,\vp_h)$ is a Palais-Smale sequence for $\S_k$ at level $c$ in a given connected component of $\mathcal M$. Then the following 
hold: 
\begin{enumerate}
\item Set $\mu_h:=\tau(\gamma_h)$ for every $h\in \N$. If $T_h\to 0$, then 
$$\int_0^1 |\dot \gamma_h+\vp_hZ(\gamma_h)|^2\, dt \to 0, \qquad e(\mu_h)\to 0.$$
\item If $0<T_*\leq T_h\leq T^*<+\infty$ for every $h\in \N$, then  $(\gamma_h,T_h,\vp_h)$ admits a converging subsequence.
\end{enumerate}
\label{lem:palaissmale}
\end{lem}

\begin{proof} 
  We start proving (1). The first assertion follows trivially from
  \eqref{2}. We now show that $e(\mu_h)\to 0$.  For every $h\in \N$ we
  consider the splitting
  $$\dot \gamma_h= \zeta_h + \<\dot \gamma_h,Z(\gamma_h)\>Z(\gamma_h),$$
  with $\zeta_h\in \ker \alpha$, and using again \eqref{2} we get
  \begin{align*}
    2kT_h^2 + o(T_h^2) & = \int_0^1 \Big |\zeta_h+\big (\<\dot \gamma_h,Z(\gamma_h)\>+\vp_h\big )Z(\gamma_h)\Big |^2 \, dt \\
    & = \int_0^1 |\zeta_h|^2\, dt + \int_0^1 \big (\<\dot
    \gamma_h,Z(\gamma_h)\>+\vp_h\big)^2 \, dt.
  \end{align*}
  In particular,
$$e(\zeta_h)=\int_0^1 |\zeta_h|^2\, dt = o(1).$$
This shows the claim, as by construction $d\tau$ is an isometry on $\ker \alpha$.

We now prove (2). Since the $T_h$'s are uniformly bounded and bounded away from zero, by Lemma \ref{lem:Tvp} we have that 
also the $\vp_h$'s are uniformly bounded, i.e. there exists $b\in \R$ such that $|\vp_h|\leq b$ for every $h\in\N$. Therefore, up to 
passing to a subsequence, we can assume that $T_h\rightarrow \bar T$ and $\vp_h\rightarrow \bar \vp$ for $h\to +\infty$. 
Moreover, using \eqref{1} and \eqref{3} we get 
\begin{align*}
c+1  &\geq \frac{1}{2T_h}\int_0^1 |\dot \gamma_h+\vp_h Z(\gamma_h)|^2\, dt - \vp_h + kT_h\\
        &= \frac{1}{2T_h}\int_0^1 |\dot \gamma_h|^2\, dt + \frac{\vp_h}{T_h}\int_0^1\<\dot \gamma_h,Z(\gamma_h)\>\, dt + \frac{\vp_h^2}{2T_h} - \vp_h + kT_h\\
        &= \frac{1}{2T_h}\int_0^1 |\dot \gamma_h|^2\, dt -\frac{\vp_h^2}{2T_h} + kT_h + o(1),
        \end{align*}
from which we deduce that, up to neglecting finitely many $h\in \N$,
$$\int_0^1 |\dot \gamma_h|^2\, dt \leq 2T_h \left (c+2 + \frac{\vp_h^2}{2T_h} -kT_h\right )\leq 2T^*\left (c+2 + \frac{b^2}{2T_*} \right ).$$

It follows that the family $\{\gamma_h\}\subseteq H^1(S^1,E)$ is $\frac 12$-H\"older-equicontinuous and hence by the Ascoli-Arzel\'a theorem 
it converges (up to a subsequence) uniformly to an element $\gamma\in C^0(S^1,E)$. Now one argues exactly as in \cite[Lemma 5.3]{Abbondandolo:2013is}
to conclude that actually $\gamma_h\to\gamma$ strongly in $H^1$.
\end{proof}

\subsection{Properties of $\S_k$ close to fiberwise rotations.} In
this subsection we study the properties of the functional $\S_k$ close
to rotations on the fibers; in particular we show that fiberwise
rotations are in some sense local minima of $\S_k$.  This generalizes
to our setting a similar well-known statement in the classical
Lagrangian setting (see for instance
\cite{Abbondandolo:2013is,Contreras:2006yo}) saying that constant
loops are ``local minima'' for the free-period Lagrangian action
functional. The contents of this section will be then used in the next
one to associate with the functional $\S_k$ a complete negative
gradient flow by truncating gradient flow-lines which approach
fiberwise rotations.

Thus, suppose that the loop $\gamma_f:S^1\rightarrow E$ satisfies $\dot \gamma_f=-\vp Z(\gamma_f).$
Clearly, $\vp \in2\pi\Z$. Assume that $\vp =2\pi a$, for some $a\in \Z$, and notice that
\begin{equation}
\S_k(\gamma_f,T,2\pi a) = - 2\pi a + kT>-2\pi a
\label{actionoffibers}
\end{equation}
converges to $-2\pi a$ for $T\to 0$. For $\delta >0$ we now define the set 
$$\mathcal V_\delta := \Big \{(\gamma,T,\vp)\in \mathcal M\ \Big |\ \int_0^1 |\dot \gamma + \vp Z(\gamma)|^2\, dt < \delta\Big \}.$$

Our first goal is to show that, for $\delta >0$ sufficiently small, the value of $\vp$ has to be close to $2\pi \Z$ for every element in $\mathcal V_\delta$.

\begin{lem}
If $(\gamma,T,\vp)\in \mathcal V_\delta$, then $\vp \in (2\pi a - \sqrt{\delta},2\pi a+\sqrt{\delta})$ for some $a\in \Z$.
\label{boundonvp}
\end{lem}

\begin{proof}
If $(\gamma,T\vp)\in\mathcal V_\delta$, then $\gamma$ satisfies $\dot \gamma = -\vp Z(\gamma) + \eta,$
for some $\eta$ such that 
$$\left (\int_0^1 |\eta|\, dt\right)^2 \leq \int_0^1 |\eta |^2 \, dt = \int_0^1 |\dot \gamma+\vp Z(\gamma)|^2 \, dt <\delta.$$
We now consider $\mu (t):= e^{i\vp t}\gamma(t)$ and compute 
$$\dot \mu = \vp Z(\mu) + e^{i\vp t}\dot \gamma=\vp Z(\mu)+e^{i\vp t} \big (-\vp Z(\gamma)+\eta\big )=e^{i\vp t}\eta.$$
If we denote with $d(\cdot,\cdot)$ the distance on $E$ induced by the Riemannian metric $g^\alpha$, then from the computation above it follows that 
$$d(\mu(1),\mu(0))\leq \int_0^1 |e^{i\vp t}\eta|\, dt < \sqrt{\delta};$$
moreover, $\mu(0)=\gamma(0)=\gamma(1)=e^{-i\vp}\mu(1)$. This implies that
\begin{equation*}
d(\mu(0),e^{-i\vp}\mu(0))=d(e^{-i\vp} \mu(1),e^{-i\vp}\mu(0))=d(\mu(1),\mu(0))< \sqrt{\delta}
\end{equation*}
and  hence trivially $\vp\in (2\pi a-\sqrt{\delta},2\pi a +\sqrt{\delta})$ for some $a\in \Z$.
\end{proof}

By the lemma above we easily get that $\mathcal V_\delta$ is the disjoint union of the sets $\mathcal V_\delta^a:= \mathcal V_\delta \cap \{\vp \in (2\pi a - \sqrt{\delta}, 2\pi a+\sqrt{\delta})\}$.
Furthermore, each set $\mathcal V_\delta^a$ contains only the fiberwise rotations given by $(\gamma_f,T,2\pi a)$. Our next step will be to show that 
the value of $\S_k$ on $\partial \mathcal V_\delta^a$ is bounded away from $-2\pi a$ by a positive constant. 

\begin{lem}
For $\delta>0$ small enough there exists $\epsilon>0$ such that, for all $a\in \Z$, 
$$\inf_{\mathcal V_\delta^a} \S_k = -2\pi a,\quad \inf_{\partial \mathcal V_\delta^a} \S_k > - 2\pi a + \epsilon.$$
\label{boundactiononvdelta}
\end{lem}
\vspace{-6mm}
\begin{proof}
For every $(\gamma,T,\vp)\in \partial \mathcal V_\delta^a$ we readily compute
$$\S_k(\gamma,T,\vp)=\frac{\delta}{2T} - \vp + kT \geq \sqrt{2k}\sqrt{\delta} - \vp \geq \sqrt{2k}\sqrt{\delta} - 2\pi a - \sqrt{\delta},$$
where in the penultimate inequality we have used minimization in the variable $T$, whilst in the last one we have used Lemma \ref{boundonvp}. The thesis follows 
as $\epsilon:=(\sqrt{2k}-1)\sqrt{\delta}$ is positive for $k>1/2$.
\end{proof}

\noindent By Equation \eqref{actionoffibers} we can easily find $T_0$ such that 
\begin{equation}
\S_k(\gamma_f,T,2\pi a)\in (-2\pi a, -2\pi a + \epsilon/4)
\label{estimateforfiberwiserotations}
\end{equation}
for every $T\in (0,T_0]$ and every $a\in \Z$. Observe that the fibers of $E$ might be contractible, as the example of the Hopf fibration $S^3\to S^2$ shows. 
However, fiberwise rotations with different winding number, say $(\gamma_f,T,2\pi a)$ and $(\gamma_f',T',2\pi a')$ with $a\neq a'\in \Z$ and $T,T'\leq T_0$, are not contained in the same connected component of
$$\big \{\S_k< \max \{-2\pi a + \epsilon,-2\pi a'+\epsilon\}\big \}$$
as every path from $(\gamma_f,T,2\pi a)$ to $(\gamma_f',T',2\pi a')$ has to intersect $\partial \mathcal V_\delta$, being the two fiberwise rotations in different 
connected components of $\mathcal V_\delta$. 

Finally we notice that, combining the discussion above with Lemma \ref{lem:palaissmale},(i) we obtain the following statement for Palais-Smale sequences with 
$T_h$ going to zero.
 
\begin{cor}
Let $(\gamma_h,T_h,\vp_h)$ be a Palais-Smale sequence for $\S_k$ at level $c$ in a given connected component of $\mathcal M$ such that $T_h\to 0$. Then $c=2\pi a$ 
for some $a\in \Z$ and $(\gamma_h,T_h,\vp_h)$ eventually enters the set $\{\S_k< -2\pi a + \epsilon/4\}\cap \mathcal V_\delta^a$.
\label{cor:palaissmale}
\end{cor}

\begin{proof}
Fix $\delta>0$. By  \eqref{2} we have that $(\gamma_h,T_h,\vp_h)\in\mathcal V_{2kT_h^2+o(T_h^2)}$ for
every $h$. In particular $(\gamma_h,T_h,\vp_h)\in\mathcal V_\delta$ for $h$ large enough. 
Furthermore, by Lemma \ref{boundonvp}, 
$$\vp_h\in \big (2\pi a_h -\sqrt{2k}T_h+o(T_h),2\pi a_h + \sqrt{2k}T_h+o(T_h)\big ),$$ 
for some $a_h\in \Z$. It follows that
\begin{align*}
  \S_k(\gamma_h,T_h,\vp_h) &=\frac{1}{2T_h}\int_0^1 |\dot \gamma_h + \vp_hZ(\gamma_h)|^2\, dt - \vp_h + kT_h\\
  & \leq 2kT_h-2\pi a_h +\sqrt{2k}T_h + o(T_h)< -2\pi a_h + \epsilon/4
\end{align*}
for $h$ large enough. On the other hand
$$\S_k(\gamma_h,T_h,\vp_h)\geq 2kT_h - 2\pi a_h - \sqrt{2k}T_h + o(T_h)\geq -2\pi a_h$$
for $h$ large enough, as $k>1/2$. Since $\S_k(\gamma_h,T_h,\vp_h)\to c$ we might
conclude that there exists some $a\in \Z$ such that $a_h=a$ for every
$h$ large enough.  In particular $c=-2\pi a$, $\vp_h\to 2\pi a$ and,
combining the estimates above,
$$(\gamma_h,T_h,\vp_h)\in \big \{\S_k\in [-2\pi a, -2\pi a + \epsilon/4)\big \}\cap \mathcal V_\delta^a$$
for every $h$ large enough, as we wished to prove.
\end{proof}


\subsection{A truncated negative gradient flow for $\S_k$.} \label{atruncatedgradient} Consider the bounded vector field 
\begin{equation}
X_k:= \frac{-\nabla \S_k}{\sqrt{1+|\nabla \S_k|^2}}
\label{boundedvf}
\end{equation}
conformally equivalent to $-\nabla \S_k$, where the gradient of $\S_k$ is made with respect to the Riemannian metric $g_{\mathcal M}$ on $\mathcal M$
and $|\cdot|$ is the norm induced by $g_{\mathcal M}$.

Clearly, the only source of non completeness for the flow $\Phi_k$ induced by $X_k$ is given by flow-lines on which the variable $T$ goes to zero. With the next
lemma we show that such flow-lines have to approach fiberwise rotations. 

\begin{lem}
Suppose $u:[0,R)\to \mathcal M, u(r)=(\gamma(r),T(r),\vp(r))$ is a maximal flow-line of $\Phi^k$. Then there exist $a\in \Z$ and $\{r_h\}_{h\in \N}$ such that 
$r_h\uparrow R$ and 
$$\int_0^1 |\dot \gamma(r_h)+\vp(r_h)Z(\gamma(r_h))|^2\, dt \to 0, \quad \vp(r_h)\to 2\pi a, \quad \S_k(u(r_h))\to -2\pi a.$$
\label{noncompleteness}
\end{lem}
\vspace{-7mm}
\begin{proof}
  Since $\liminf_{r\to R} T(r)= 0$ we can find a sequence
  $\{r_h\}_{h\in \N}$ such that $r_h\uparrow R$, $T(r_h)\to 0$ and
  $T'(r_h)\leq 0$ for every $h\in\N$. Using \eqref{dskdT} we get that
$$0\geq \rho_h T'(r_h)=- \frac{\partial \S_k}{\partial T}(u(r_h))=\frac{1}{2T(r_h)^2}\int_0^1 |\dot{\gamma(r_h)}+\vp(r_h)Z(\gamma(r_h))|^2\, dt - k,$$
where $\rho_h:=\sqrt{1+|\nabla \S_k(\gamma_h)|^2}$,
and hence
\begin{equation}
\int_0^1 |\dot{\gamma(r_h)}+\vp(r_h)Z(\gamma(r_h))|^2\, dt\leq 2k T(r_h)^2.
\label{integralandT}
\end{equation}
This proves the first assertion. We now use Lemma \ref{boundonvp} to infer that 
$$\vp(r_h)\in\big (2\pi a(r_h)-\sqrt{2k} T(r_h),2\pi a(r_h)+\sqrt{2k}T(r_h)\big )$$ 
for some $a(r_h)\in \Z$ and compute 
\begin{align*}
\S_k(u(r_h)) &=\frac{1}{2T(r_h)}\int_0^1 |\dot{\gamma(r_h)}+\vp(r_h)Z(\gamma(r_h))|^2\, dt - \vp(r_h)+kT(r_h)\\
	            &\leq 2 k T(r_h) - 2\pi a (r_h) + \sqrt{2 k} T(r_h)\\
	            &< -2 \pi a(r_h) + \epsilon
\end{align*}
for $h$ large enough, where $\epsilon$ is the constant given by Lemma \ref{boundactiononvdelta}. On the other hand,
$$\S_k(u(r_h)) \geq -2\pi a(r_h),$$
for the infimum of $\S_k$ on $\mathcal V_{2kT(r_h)^2}^{a(r_h)}$ is $-2\pi a(r_h)$. This shows that 
$$(\gamma(r_h),T(r_h),\vp(r_h))\in \big \{\S_k <-2\pi a(r_h)+\epsilon)\big \}\cap \mathcal V_{2kT(r_h)^2}^{a(r_h)}$$
for every $h$ large enough. Since $r \mapsto S_k\circ u(r)$ is non-increasing we conclude that there exist $\delta>0$, $a\in \Z$ and $\bar h\in \N$ such that 
$u(r)\in \V^a_\delta$ for every $r\geq r_{\bar h}$. In particular, $a(r_h)=a$ for every $h$ large enough and
hence $\vp(r_h)\to 2\pi a$, $\S_k(u(r_h))\to -2\pi a$, as we wished to
show.
\end{proof}

Using Lemma \ref{noncompleteness} it is now easy to get from $\Phi^k$ a complete flow. Namely, we stop flow-lines which enter the connected component of the sublevel set 
$\{\S_k<-2\pi a +\epsilon /2\}$ containing the fiberwise rotations $(\gamma_f,T,2\pi a)$, $T\leq T_0$.
With slight abuse of notation, we denote the complete flow also with $\Phi^k$.


\section{Proof of Theorem \ref{thm:main}}
\label{proofoftheorem}


In this section, building on the results of the previous ones, we prove Theorem \ref{thm:main}. In order to show the existence of critical points of $\S_k$, we will use the topological assumption on $M$ to build a suitable (non-trivial) minimax class on the Hilbert manifold $\mathcal M$ and a corresponding minimax function. We will then show that
such a minimax function yields critical points of $\S_k$ for almost every $k>\frac 12$.

The first step in this direction is therefore to show that the assumption on the topology of $M$ is preserved when passing to the 
$S^1$-bundle. As a precursor, we recall the relation between the homotopy groups of $E$ and the ones of $M$. 
\begin{lem}
\label{lem:homotopygroups}
 The maps $\pi_\ell(\tau):	\pi_\ell(E)\to\pi_\ell(M)$, $\ell\in\N$, of homotopy groups induced by the $S^1$-bundle $\tau:E\to M$ satisfy:
  \begin{itemize}
  \item $\pi_\ell(\tau)$ is an isomorphism for $\ell\geq 3$.
  \item $\pi_1(\tau)$ is surjective and its kernel is isomorphic to
    $\Z/m\Z$. 
  \item $\pi_2(\tau)$ is injective and $\pi_2(M)\cong m\Z\oplus \mathrm{im}\,
    \pi_2(\tau)$.
  \end{itemize}
Here $m$ is defined by the relation
$ \{ \<e,A\>\ |\ A \in H_2^S(M)\}=m\Z,$
where $H_2^S(M) \subset H_2(M;\Z)$ denotes the image of the Hurewicz map $\pi_2(M) \to H_2(M;\Z)$, $e \in H^2(M)$ the Euler class of $E \to M$ and $\<e,A\>$ the dual pairing. 
\end{lem}
\begin{proof}
  Consider the long exact homotopy sequence
  \[
  \dots \to \pi_{\ell}(S^1)\to \pi_\ell(E) \stackrel{\pi_\ell(\tau)}{\longrightarrow}
  \pi_\ell(M)\to \pi_{\ell-1}(S^1)\to \dots\,,
  \]
  This readily shows the first assertion. For $\ell=2$ the connecting homomorphism fits into the commuting square
  \[
  \xymatrix{ 0 \ar[r] & \pi_2(E)\ar[r] &  \pi_2(M) \ar[r]\ar[d]& \pi_1(S^1)\ar[d] \ar[r]& \pi_1(E)\ar[r] & \pi_1(M)\ar[r] &0\\
                       & & H^S_2(M;\Z)\ar[r]& \Z & & & }
  \]
  where the vertical arrows are the Hurewicz map and the canonical
  isomorphism and the horizontal maps are the connecting homomorphism
  and the pairing with the Euler class. This readily implies the other two statements.
\end{proof}

\begin{lem}
If $M$ is non-aspherical, then $E$ is non-aspherical.
\label{lem:enonaspherical}
\end{lem}
\begin{proof}
  Recall that, by Lemma \ref{lem:homotopygroups}, $\pi_\ell(M)$ is
  isomorphic to $\pi_\ell(E)$ for every $\ell\geq 3$. In particular,
  if $\pi_\ell(M)\neq \{0\}$ for some $\ell\geq 3$, then also $\pi_\ell(E)\neq
  \{0\}$.  Thus, we are left with the case $\pi_2(M)\neq \{0\}$ and
  $\pi_\ell(M)=\{0\}$, for every $\ell\geq 3$. Assume by contradiction
  that $E$ is aspherical, i.e.\ $\pi_\ell(E) =0$ for all $\ell \geq
  2$. But then again by Lemma~\ref{lem:homotopygroups} we conclude
  that $\pi_2(M) \cong \Z$ and thus the universal cover of $M$
  satisfies
$$\pi_2(\widetilde M)\cong \Z,\qquad \pi_\ell(\widetilde M)=\{0\},  \qquad \forall\ \ell \neq 2.$$
In particular, $\widetilde M$ is homotopy equivalent
(c.f. \cite{Eilenberg:1945,Eilenberg:1950}) to the Eilenberg-Maclane
space $K(\Z,2)\cong\C\PP^\infty$.  This is however not possible for a
finite-dimensional manifold, since $H_{2j}(\C\PP^\infty,\Z)\cong\Z$
for every $j\in \N$.
\end{proof}

By the lemma above there exists a non-zero element $\mathfrak u\in
\pi_\ell(E)$ for some $\ell \geq 2$. Notice that, by Lemma
\ref{lem:homotopygroups}, $\pi_\ell(\tau)(\mathfrak u)\neq 0\in
\pi_\ell(M)$. With $\mathfrak u$ we now associate a suitable class of
paths in $\mathcal M_0$ over which we will perform the minimax
procedure.

We start observing that any continuous map
$$f:(B^{\ell-1},S^{\ell-2})\to (H^1(S^1,E),E),$$ 
defines a continuous map $v(f):S^\ell\to E$ (c.f. for instance
\cite[Proof of Theorem 2.4.20]{Klingenberg95}); here, with slight
abuse of notation, we have denoted with $E$ the set of constant loops
in $H^1(S^1,E)$. Conversely, every regular map $v:S^\ell\to E$,
defines a continuous map
$$f(v):(B^{\ell-1},S^{\ell-2})\to (H^1(S^1,E),E).$$

Notice furthermore that, by \eqref{estimateforfiberwiserotations} we can find a positive constant $T_0>0$ such that $\max \S_k|_{E_{T_0,0}}\leq \epsilon/4$, where $\epsilon>0$ 
is the constant given by Lemma \ref{boundactiononvdelta} and 
$$E_{T_0,0}:=\bigcup_{T\leq T_0} E\times \{T\}\times \{0\}.$$

Now set
$$\mathcal P:= \Big \{u=(f,T,\vp):(B^{\ell-1},S^{\ell-2})\to (\mathcal M_0,E_{T_0,0}) \Big |\ [v(f)]=\mathfrak u\Big \}.$$

We readily see that $\mathcal P\neq \emptyset$, since $(f(v),T,\vp)\in\mathcal P$ for any smooth map $v:S^\ell\to E$ such that $[v]=\mathfrak u$. Moreover, $\mathcal P$
is by construction invariant under the complete flow $\Phi^k$ defined in Subsection \ref{atruncatedgradient}. The last property of $\mathcal P$
we will need is that every element $u\in \mathcal P$ has to intersect $\partial \mathcal V_\delta$ (more precisely, $\partial \mathcal V_\delta^0$). 
Indeed, if $u(\cdot)\subseteq \mathcal V_\delta$, then $u(\cdot)$ would have to be entirely contained in $\mathcal V_\delta^0$ (simply because 
$\vp(S^{l-2})=0$ and $\mathcal V_\delta$ is the disjoint union of the sets $\mathcal V_\delta^a$, $a\in 2\pi \Z$) and hence, using the splitting 
$$\dot {f(s)} =\zeta(s) + \<\dot{f(s)},Z(f(s))\>Z(f(s))$$ 
with $\zeta(s) \in \ker \alpha$, we would get $e(\zeta(s))<\delta$ for every $s\in [0,1]$. In particular, since by construction $d\tau$ is 
an isometry on $\ker \alpha$, we would have that 
$e(\tau\circ f(s))<\delta$, for all $s\in [0,1].$
This would imply that $[\tau \circ f]=0\in \pi_\ell(M)$ (see for instance \cite[Section 2.4]{Klingenberg95}), in contradiction with our assumption 
(recall indeed that $\pi_\ell(\tau)(\mathfrak u)\neq0$).

We now define the minimax function 
$$c:(\frac 12,+\infty)\to (0,+\infty),\qquad c(k):= \inf_{u\in\mathcal P}\max_{\zeta\in B^{\ell-1}} \S_k(u(\zeta)).$$

Observe that $c(k)\geq \epsilon$, for every $u\in \mathcal P$ has to
intersect $\partial \mathcal V_\delta^0$. However, this is not enough
to exclude that the periods of a Palais-Smale sequence for $\S_k$ converge
to zero as $h\to +\infty$, as it might well be that $c(k)=2\pi a$ for
some $a\in \Z$.  For that we will need the piece of additional
information given by the following lemma.

\begin{lem} Let $u$ be any element of $\mathcal P$.
Suppose that $\zeta^*\in B^{\ell-1}$ is such that
\begin{equation}
\S_k(u(\zeta^*))\geq \max_{B^{\ell-1}}\S_k\circ u - \epsilon/2.
\label{almostmaximum}
\end{equation}
Then $u(\zeta^*)\notin \cup_{a\in \Z} (\{\S_k< -2\pi a + \epsilon/2\}\cap \mathcal V_\delta^a$).
\label{lem:almostmaximum}
\end{lem}
\begin{proof}
Suppose by contradiction that there exists $a\in\Z$ such that $u(\zeta^*)\in \{\S_k< -2\pi a + \epsilon/4\}\cap \mathcal V_\delta^a$. Since $u\in \mathcal P$ there 
exists $\zeta\in B^{\ell-1}$ such that $u(\zeta)\in \partial \mathcal V_\delta^a$. Using Lemma \ref{boundactiononvdelta} we now readily see that 
\begin{align*}
  \max_{B^{\ell-1}} \S_k \circ u - \S_k(u(\zeta^*))&\geq \S_k(u(\zeta)) - \S_k(u(\zeta^*))\\
  &> -2\pi a + \epsilon + 2\pi a - \epsilon/2 = \epsilon/2,
\end{align*}
in contradiction with \eqref{almostmaximum}.
\end{proof}

Clearly, the function $c(\cdot)$ is monotonically increasing in $k$
and hence almost everywhere differentiable.  With the next proposition
we show that we can find Palais-Smale sequences
$(\gamma_h,T_h,\vp_h)\subseteq \mathcal M_0$ for $\S_k$ with $T_h$'s
bounded away from zero and uniformly bounded, as soon as $k$ is a
point of differentiability for $c(\cdot)$. The proof goes along the
line of \cite[Lemma 8.1]{Abbondandolo:2013is} (see also
\cite[Proposition 7.1]{Contreras:2006yo}) and \cite[Proposition
4.1]{Asselle:2014hc} and relies on the celebrated \textit{Struwe
  monotonicity argument} \cite{Struwe:1990sd}. This concludes the
proof of Theorem \ref{thm:main} in virtue of Lemma
\ref{lem:palaissmale},(2).

\begin{prop}
  Let $k^*$ be a point of differentiability for $c(\cdot)$. Then there
  exists a Palais-Smale sequence $(\gamma_h,T_h,\vp_h)\subseteq
  \mathcal M_0$ for $\S_{k^*}$ with $T_h$ bounded and bounded away from
  zero.
\label{prop:struwe}
\end{prop}

\begin{proof}
Let $M$ be a right linear modulus of continuity for $c(\cdot)$ at $k^*$. This means that for all $k\geq k^*$ sufficiently close to $k^*$ we have
\begin{equation}
c(k)-c(k^*)\leq M (k-k^*).
\label{rightmodulus}
\end{equation}

Consider a sequence $k_n\downarrow k^*$ and set
$b_n:=k_n-k^*\downarrow 0$. Without loss of generality we can suppose
that \eqref{rightmodulus} holds for every $n\in \N$. For every
$n\in\N$ pick an element $u_n\in \mathcal P$ such that
$$\max_{\zeta\in B^{\ell-1}} \S_{k_n}(u_n(\zeta))<c(k_n)+b_n\leq c(k^*)+(M+1)b_n$$
If $\zeta\in B^{\ell-1}$ is such that $\S_{k^*}(u_n(\zeta))\geq c(k^*)-b_n$, then using \eqref{rightmodulus} we get 
$$T_n(\zeta) = \frac{\S_{k_n}(u_n(\zeta))-\S_{k^*}(u_n(\zeta))}{b_n} \leq M+2.$$
It follows that, for all $n\in \N$, $u_n$ is contained in
$$\{\S_{k^*}\leq c(k^*)-b_n\}\cup \Big \{\S_{k^*}\in\big (c(k^*)-b_n,c(k^*)+(M+1)b_n\big ), \ T\leq M+2 \Big \}.$$
For every $r\in [0,1]$ and every $n\in \N$ we now define $u^r_n\in
\mathcal P$ by
$$u_n^r(\zeta):=\Phi^{k^*}_r(u_n(\zeta)),\quad \forall \zeta\in B^{\ell-1},$$
where $\Phi^{k^*}_r$ is the complete flow defined in Subsection
\ref{atruncatedgradient}. Namely, for $\zeta\in B^{\ell-1}$ fixed,
$r\mapsto u^r_n(\zeta)$ is the flow-line of $\Phi^{k^*}$ starting at
$u_n(\zeta)$.  Since $\S_{k^*}$ is non-increasing along flow-lines of
$\Phi^{k^*}$ and the vector-field generating $\Phi^{k^*}$ has norm less than
or equal to one we have that, for all $r\in [0,1]$ and every $n\in\N$,
$$u_n^r \subset \{\S_{k^*}\leq c(k^*)-b_n\}\cup \Big \{\S_{k^*}\in\big (c(k^*)-b_n,c(k^*)+(M+1)b_n\big ), \ T\leq M+3 \Big \}.$$
For any  $\zeta\in B^{\ell-1}$ we now have two possibilities:
\begin{enumerate}[i)]
\item $\S_{k^*}(u^1_n(\zeta))\leq c(k^*)-b_n$.
\item $\S_{k^*}(u^r_n(\zeta))\in (c(k^*)-b_n,c(k^*)+(M+1)b_n)$, for every
  $r\in [0,1]$.
\end{enumerate}
If $\zeta\in B^{\ell-1}$ satisfies the second alternative then we have
\begin{align*}
  \S_{k^*}(u^r_n(\zeta))>c(k^*)-b_n&>\max_{B^{\ell-1}} \S_{k^*}\circ u_n^r - (M+2)b_n\\
  &>\max_{B^{\ell-1}} \S_{k^*}\circ u_n^r - \epsilon/2
\end{align*}
for every $n\in \N$ large enough. Therefore, by Lemma
\ref{lem:almostmaximum}, $u^r_n(\zeta)\notin \cup_{a\in \Z} (\{\S_{k^*}<
-2\pi a + \epsilon/2\}\cap \mathcal V_\delta^a)$ for every $r\in
[0,1]$ and every $n\in \N$ large enough. In other words, $r\mapsto
u^r_n(\zeta)$ is a genuine flow-line for the flow of the vector field
$X_{k^*}$ in \eqref{boundedvf}. We now claim that there exists a
Palais-Smale sequence for $\S_{k^*}$ contained in
$$\mathfrak K:= \{T\leq M+3\}\setminus \bigcup_{a\in \Z} (\{\S_{k^*}< -2\pi a + \epsilon/2\}\cap \mathcal V_\delta^a).$$
Notice that this completes the proof, since such a Palais-Smale
sequence has $T_h$ trivially uniformly bounded and bounded away from
zero by Corollary \ref{cor:palaissmale}.

Thus, suppose by contradiction that $\S_{k^*}$ does not have Palais-Smale
sequences contained in $\mathfrak K$, then we can find $\rho>0$ such
that $|X_{k^*}|\geq \rho$ on $\mathfrak K$. If $\zeta\in B^{\ell-1}$
satisfies the alternative ii) above, then we compute
$$(M+2)b_n>\S_{k^*}(u_n(\zeta))- \S_{k^*}(u_n^1(\zeta))=\int_0^1 |X_{k^*}|^2 dr \geq \rho^2,$$
which is impossible for $n$ large. It follows that, for $n$ large
enough, every $\zeta\in B^{\ell-1}$ satisfies the alternative i), that
is
$$\max_{B^{\ell-1}} \S_{k^*}\circ u_n^1\leq c(k^*)-b_n,$$
in contradiction with the definition of $c(k^*)$.
\end{proof}

\noindent \textbf{Acknowledgments.} We warmly thank Prof. Felix Schlenk and
Prof. Alberto Abbondandolo for helpful discussions.  L. A. is
partially supported by the DFG grant AB 360/2-1 ``Periodic orbits of
conservative systems below the Ma$\tilde{\text{n}}$\'e critical energy
value". We are also in debt to Gabriele Benedetti for his precious suggestions about a previous version of the draft.

\bibliography{_biblio}
\bibliographystyle{plain}

\end{document}